\documentclass[12pt]{amsart}
\usepackage{enumerate}

\newtheorem{theorem}{Theorem}[section]
\newtheorem{lemma}[theorem]{Lemma}
\newtheorem{proposition}[theorem]{Proposition}
\newtheorem{corollary}[theorem]{Corollary}

\theoremstyle{definition}
\newtheorem{definition}[theorem]{Definition}
\newtheorem*{remark}{Remark}
\newtheorem*{theorem*}{Theorem}

\DeclareMathOperator{\id}{Id}
\newcommand{\norm}[1]{\Vert #1 \Vert}
\newcommand{\nnorm}[1]{\lvert\!|\!| #1|\!|\!\rvert}
\newcommand{\E}{{\mathbb E}}
\newcommand{\N}{{\mathbb N}}
\newcommand{\ZZ}{{\mathbb Z}}
\newcommand{\K}{{\mathcal K}}
\newcommand{\CI}{{\mathcal I}}
\newcommand{\CA}{{\mathcal A}}
\newcommand{\CF}{{\mathcal F}}
\newcommand{\CW}{{\mathcal W}}
\newcommand{\CX}{{\mathcal X}}
\newcommand{\CY}{{\mathcal Y}}
\newcommand{\CZ}{{\mathcal Z}}

\begin{document}

\title[]{Multiple recurrence for two commuting transformations}
\author{Qing Chu}

\address{Universit\'{e} Paris-Est, Laboratoire d'Analyse et de Math\'{e}matiques
Appliqu\'{e}es, UMR CNRS 8050, 5 bd Descartes, 77454 Marne la Vall\'{e}e Cedex 2, France}

\email{qing.chu@univ-mlv.fr}

\subjclass[2000]{37A05, 37A30}

\date{\today}%

\keywords{Multiple recurrence, commuting transformations, magic system, ergodic
seminorms}

\begin{abstract}
This paper is devoted to a study of the multiple recurrence of two commuting transformations. We derive a result which is similar but not identical to that of one single transformation established by Bergelson, Host and Kra. We will use
the machinery of ``magic systems" established
recently by B. ~Host for the proof.
\end{abstract}

\maketitle

\section{Introduction}

\subsection{History and results}
Let $(X,\CX,\mu,T)$ be an invertible measure preserving system, and $A$ be a set of positive measure. The Khintchine's Recurrence Theorem~\cite{K} states that for every $\epsilon>0$, the set
$$\{n\in\ZZ\colon\ \mu(A\cap T^n A)>\mu(A)^2-\epsilon\}$$
is syndetic. More recently, Furstenberg~\cite{Fur0} proved a Multiple Recurrence Theorem, showing that under the same assumptions,  the set $$\{n\in \ZZ\colon\ \mu(A\cap T^n A\cap T^{2n}A\cap\cdots\cap T^{kn}A)>0\}$$ is syndetic for every integer $k\geq 1$. Aiming at a simultaneous extension of Khintchine's and
Furstenberg's Recurrence theorems, Bergelson, Host and
Kra~\cite{BHK} established the following result:

\begin{theorem*}[Bergelson, Host, Kra]
 Let $(X,\CX,\mu,T)$ be an ergodic invertible measure preserving system and $A\in\CX$ with $\mu(A)>0$. Then for every $\epsilon>0$, the sets
$$\{n\in \ZZ\colon\ \mu(A\cap T^n A\cap T^{2n}A)>\mu(A)^3-\epsilon\}$$ and
$$\{n\in \ZZ\colon\ \mu(A\cap T^n A\cap T^{2n}A\cap T^{3n}A)>\mu(A)^4-\epsilon\}$$
are syndetic. 
\end{theorem*}

We recall that a subset $E$ of $\ZZ$ is said to be \textit{syndetic}
if there exists an integer $N>0$ such that $E\cap [M,M+N)\neq
\emptyset$ for every $M\in \ZZ$.

It was shown in~\cite{BHK} that an analogous result fails both if we remove the assumption of ergodicity and if for longer arithmetic progressions.

Furstenberg and Katznelson~\cite{FurK} generalized Furstenberg's
Recurrence theorem to commuting transformations. It is therefore natural to ask the question
if a result analogous to the above theorem can be established for commuting transformations. We
prove the following result regarding the case of two transformations:

\begin{theorem}\label{mmaina}
Let $(X,\CX, \mu)$ be a probability space, and $T_1, T_2$ be two commuting invertible measure preserving transformations. Assume that $(X,\CX,\mu,T_1,T_2)$ is ergodic. Let $A\in \CX$ with $\mu(A)>0$. Then for every $\epsilon>0$, the set
$$\{n\in\ZZ\colon\ \mu(A\cap T_1^n A\cap T_2^n A)> \mu(A)^4-\epsilon\}$$
is syndetic.
\end{theorem}
We remark that, by the same counterexample as in~\cite{BHK}, the hypothesis of ergodicity is necessary for the theorem.

The fundamental difference with the case of a single transformation
is that the exponent $4$ can not be replaced by $3$:
\begin{theorem}\label{counterex}
For every $0<c\leq 1$, there exist a probability space $(X,\CX,\mu)$, with two commuting invertible measure preserving transformations $T_1,T_2$ such that $(X,\CX,\mu,T_1,T_2)$ is ergodic, and a measurable set $A\in \CX$, with $\mu(A)>0$, such that
$$\mu(A\cap T_1^n A\cap T_2^n A)< c\mu(A)^3$$
for every integer $n\neq 0$.
\end{theorem}

However, we have the exponent $3$ for some class of systems:
\begin{theorem}\label{proda}
Let $(Y_1,\CY_1,\nu_1,S_1)$ and $(Y_2,\CY_2,\nu_2,S_2)$ be two
ergodic invertible measure preserving systems. Let $(X,\CX,\mu)$ be the product measure space
$(Y_1\times Y_2,\CY_1\otimes \CY_2, \nu_1\times\nu_2)$, and let
$T_1=S_1\times \id$, $T_2=\id\times S_2$. Let $A\in \CX$ with $\mu(A)>0$. Then for
every $\epsilon>0$, the set
$$\{n\in \ZZ\colon\ \mu(A\cap T_1^n A\cap T_2^n A)> \mu(A)^3-\epsilon\}$$
is syndetic.
\end{theorem}

We recall a definition:
\begin{definition}
The \textit{upper Banach density} of a subset $E$ of $\ZZ^2$ is:
$$d^*(E)=\limsup_{\substack{N_1-M_1\rightarrow \infty\\N_2-M_2\rightarrow \infty}}\frac{\mid E\cap
[M_1,N_1)\times[M_2,N_2)\mid}{(N_1-M_1)\times (N_2-M_2)}\ .$$
\end{definition}

Using a variation (see~\cite{BHK}) of Multidimensional Furstenberg's Correspondence Principle~\cite{Fur} and Theorem~\ref{mmaina}, we deduce:
\begin{corollary}
 Let $E\subset \ZZ^2$ be a subset with positive upper Banach density. Then for every $\epsilon>0$, the set
$$\{n\in \ZZ\colon\ d^*(E\cap (E+(n,0))\cap (E+(0,n)))>d^*(E)^4-\epsilon\}$$
is syndetic.
\end{corollary}

\subsection{Questions}
We address here some questions which are related to this paper and remain open.

\textbf{Question 1:} A natural question is how about the case of
three commuting transformations. We ask if there exists some integer
$s$, such that, for every ergodic system $(X,\CX, \mu,T_1,T_2,T_3)$ and every set
$A\in \CX$ with $\mu(A)>0$, the set $\{n\in\ZZ\colon \mu(A\cap T_1^n
A\cap T_2^n A\cap T_3^n A)>\mu(A)^s-\epsilon\}$ is syndetic.

\textbf{Question 2:} The Polynomial Recurrence Theorem was proved by
Bergelson and Leibman~\cite{BL}. Frantzikinakis and Kra~\cite{FK}
provided a more precise result for any family of linearly independent
integer polynomials:
\begin{theorem*}[Frantzikinakis, Kra]\label{thfk}
 Let $(X,\CX,\mu,T)$ be an invertible measure preserving system, and $p_1,\dots,p_k$ be linearly independent integer polynomials with $p_i(0)=0$ for $i=1,\dots,k$, and $A\in\CX$. Then for every $\epsilon>0$, the set
$$\{n\in \ZZ\colon\ \mu(A\cap T^{p_1(n)}A\cap\cdots\cap T^{p_k(n)}A)\geq \mu(A)^{k+1}-\epsilon\}$$
is syndetic.
\end{theorem*}

Can this result be generalized for commuting transformations? Let
$(X,\CX, \mu)$ be a probability space, and $T_1,\dots, T_k$ be commuting invertible measure preserving transformations. Let $p_1,\dots,p_k$ be as in
the above theorem, and $A\in \CX$. It is true that the set $\{n\in
\ZZ\colon\ \mu(A\cap T_1^{p_1(n)}A\cap\cdots\cap T_k^{p_k(n)}A)\geq
\mu(A)^{k+1}-\epsilon\}$ is syndetic? Very recently, Chu, Frantzikinakis, and Host~\cite{CFH} gave an affirmative answer to this question when $p_i(n)=n^{d_i},\ i=1,\dots, k$ for distinct positive integers $d_1,\dots, d_k$.

\subsection{Conventions and notation}
\subsubsection{}
As usual, we can restrict to the case that all the probability
spaces that we deal with are standard.

In general, we write $(X,\mu)$ for a probability space, omitting the
$\sigma$-algebra. When needed, the $\sigma$-algebra of a the
probability space $(X,\mu)$ is written $\CX$.

We implicitly assume that the term ``bounded function'' means
real-valued, bounded and measurable.

If $S$ is a measure-preserving transformation of a probability space
$(X,\CX,\mu)$ then we write $\CI(S)$ for the sub$-\sigma$-algebra of
$\CX$ consisting in $S$-invariant sets.

Let $(X_1,\CX_1,\mu_1)$, $(X_2,\CX_2,\mu_2)$ be two probability
spaces. Let $f_1\in L^{\infty}(\mu_1)$, and $f_2\in
L^{\infty}(\mu_2)$, we denote by $f_1\otimes f_2$ the function on
$X_1\times X_2$ given by $f_1\otimes f_2(x_1,x_2)=f_1(x_1)f_2(x_2)$.

\subsubsection{}
Throughout this paper, by a \textit{system}, we mean a probability space
$(X,\CX,\mu)$ endowed with a single or several commuting measure
preserving invertible transformations.

For a system $(X,\mu,T_1,T_2)$, a \textit{factor} is a system
$(Y,\nu,S_1,S_2)$ and a measurable map $\pi\colon X\to Y$ such that
the image $\pi (\mu)$ of $\mu$ under $\pi$ is equal to $\nu$ and
$S_i\circ \pi=\pi\circ T_i$, $i=1,2$, $\mu$-a.e.

If $f$ is an integrable function on $X$, we write $\E(f\mid Y)$ for the function on $Y$ defined by
$$\text{for all } g\in L^{\infty}(\nu),\ \ \ \int _{X}f\cdot g\circ \pi\, d\mu=\int _{Y}\E(f|Y)\cdot g\, d\nu\ .$$

\subsubsection{}
We say that \textit{the averages of some sequence $(a_n)$ converge} to some limit $L$, and we write: 
$$\lim_{N-M\rightarrow\infty}\frac{1}{N-M}\sum_{n\in [M,N)} a_n=L$$
if the averages of $a_n$ on any sequences of intervals $[M_i,N_i)$
whose lengths $N_i-M_i$ tend to infinity converge to $L$:
$$\lim_{i\rightarrow\infty}\frac{1}{N_i-M_i}\sum_{n\in[M_i,N_i)} a_n=L.$$

For $t\in \mathbb T$, we use the standard notation $e(nt)=\exp(2\pi int)$.

\subsection{Methods}
\subsubsection{}
The first ingredient used in the proofs of Theorem~\ref{mmaina} and Theorem~\ref{proda} is a
method introduced by Frantzikinakis~\cite{F}: in order to show some sequence $I_n$
is large for $n$ in a syndetic set, we show that the average of
$I_n$ over an appropriately chosen sequence of intervals converges
to a large limit. More precisely, we introduce an ergodic rotation
$(Z,\alpha)$, and we show that for some well chosen non-negative continuous
function $\chi$ on $Z$, the weighted average of $\chi(n\alpha)I_n$
converges to a large limit. This strategy is sufficient for the proof of Theorem~\ref{proda} (Section~\ref{sec:prod}).

However, for the general case
(Theorem~\ref{mmaina}), before using the first ingredient, we need first to make some reduction. We will use more elaborated tools such as the machinery of
``magic systems" introduced recently by Host~\cite{H}. 

Ever since Tao~\cite{T} proved the norm convergence of multiple ergodic averages with several commuting transformations of the form
\begin{equation}\label{retao}
\frac{1}{N}\sum_{n=1}^N T_1^n f_1\cdot\ldots\cdot T_d^n f_d\ ,
\end{equation} several other proofs were given with different approaches by Austin ~\cite{A}, Host ~\cite{H} and Towsner~
\cite{Tow}. Among these proofs, those by Austin and Host were proceeded by building an \textit{extension} of the original system with good properties, called ``pleasant'' system (using terminology from~\cite{A}) and ``magic'' system (using terminology from~\cite{H}). We will follow this idea of building a suitable extension system. 

We first prove that every ergodic system has an ergodic magic extension. Then we use this result three times, and we get an ergodic magic extension. We consider our problem on this extension system. By using the properties of magic systems related to the convergence of multiple averages, we are reduced to consider our problem on a factor of this extension system. At last, we give a description of this factor, and we find that we are in a situation very similar to that of the product case.

We will review ``magic systems" and
the corresponding properties needed in Section~\ref{sec:pre}, and give the proof of Theorem~\ref{mmaina} in Section~\ref{sec:general}.

\subsubsection{}
Another important ingredient is the following inequality. This inequality was proved in a particular case by Atkinson, Watterson, and Moran~\cite{AWM}, and it is related to a class of inequalities studied by Sidorenko~(\cite{S1},~\cite{S2}).

\begin{lemma}\label{ineq}
Let $(X,\CX,\mu)$ be a probability space, $k\geq 1$ be an
integer, and $\CX_1,\CX_2,\dots,\CX_k$ be $k$ sub-$\sigma$-algebras
of $\CX$. For any bounded non-negative function $f$ on $X$, we have
\begin{equation}\label{eq:ineq}
\int f \cdot \prod_{i=1}^k \E(f\mid \CX_i)\,d\mu\geq (\int f
d\mu)^{k+1}\ .
\end{equation}
\end{lemma}
\begin{proof}
We can restrict to the case that the function $f$ is bounded below
by some positive constant $\epsilon$. Indeed, for the general case,
it suffices to apply the inequality to the function $f+\epsilon$ and
to take the limit of both sides when $\epsilon$ tends to $0$.

We write
\begin{equation}\label{eq:ineq1}
f=\Big(f\cdot\prod_{i=1}^k \E(f\mid
\CX_i)\Big)^{1/k+1}\cdot\prod_{i=1}^k\Big(\frac{f}{\E(f\mid\CX_i)}\Big)^{1/k+1}\
.\end{equation} Let $J$ be the integral on the left hand side
of~\eqref{eq:ineq}. By the H\"{o}lder Inequality and~\eqref{eq:ineq1},
\begin{equation}
\big(\int f d\mu\big)^{k+1}\leq  J\cdot \prod_{i=1}^k\int\frac{f}{\E(f\mid
\CX_i)}d\mu\ .
\end{equation}
On the other hand, for $1\leq i\leq k$,
$$\int \frac{f}{\E(f\mid\CX_i)}d\mu=1\ ,$$
and this proves the inequality.
\end{proof}

\subsection*{Acknowledgement}
The author would like to thank her advisor, Bernard Host, for many helpful discussions and suggestions.

\section{Proof of Theorem~\ref{proda}}\label{sec:prod}
In this section, we assume that $(Y_1,\nu_1,S_1)$ and
$(Y_2,\nu_2,S_2)$ are two ergodic systems, and that
$(X,\mu,T_1,T_2)=(Y_1\times Y_2,\nu_1\times\nu_2,S_1\times
\id,\id\times S_2)$. Then $(X,\mu,T_1,T_2)$ is ergodic.

For every $n$, and all bounded functions $\eta,\varphi,\psi$ on $X$, we
define
\begin{eqnarray*}
I_n(\eta,\varphi,\psi) :&=&\int \eta \cdot T_1^n \varphi\cdot T^n_2
\psi\, d\mu
\\&=&\int \eta(y_1,y_2)\varphi(S_1^n y_1,y_2)\psi(y_1, S_2^n y_2)\, d\nu_1(y_1)d\nu_2(y_2)\ .
\end{eqnarray*}

We prove:
\begin{theorem}\label{prodf}
Let $0\leq f\leq 1$. Then for every $\epsilon>0$, the set
$$\{n\in\ZZ\colon\ I_n(f,f,f)> (\int f d\mu)^3-\epsilon\}$$
is syndetic.
\end{theorem}
Theorem~\ref{proda} follows from Theorem~\ref{prodf} with
$f=\mathbf{1}_A$.

\subsection{An ergodic rotation}
Let $(Z,\CZ,\theta,R)$ denote the common factor of of systems $Y_1$ and
$Y_2$ spanned by the eigenfunctions corresponding to the common
eigenvalues of these two systems. 

We recall that $Z$ is a compact abelian group, endowed
with a Borel $\sigma$-algebra $\CZ$ and Haar measure $\theta$. $R$
is the translation by some fixed $\alpha\in Z$. We consider $\CZ$ as
a sub-$\sigma$-algebra of both $(Y_1,\CY_1)$ and $(Y_2,\CY_2)$.

For a bounded function $\eta$ on $X$, we write
$$
\widehat{\eta}\colon=\E(\eta\mid\CZ\times \CY_2)\ \text{and}\ \widetilde{\eta}\colon=\E(\eta\mid\CY_1\times \CZ)\ .
$$

We begin with a classical lemma (see for example~\cite{Fur}, Lemma 4.18):
\begin{lemma}\label{claimprod}
$$\CI(S_1\times S_2)\subset \CZ\times \CZ.$$
\end{lemma}

\begin{proof}
For $i=1,2$, let $\K_i$ be the Kronecker factor of $(Y_i,S_i)$. It is
known that any $S_1\times S_2$-invariant function on $Y_1\times Y_2$
is measurable with respect to $\K_1\times \K_2$.

Let $\{f_i\}$ (resp. $\{g_j\}$) be an orthonormal basis of
$L^2(\K_1,\nu_1)$ (resp. $L^2(\K_2,\nu_2)$) consisting of
eigenfunctions of $S_1$ (resp. $S_2$), then the set $\{f_i\otimes
\bar g_j\}$ is an orthonormal basis of $L^2(\K_1\times
\K_2,\nu_1\times\nu_2)$.

Any $F\in L^{\infty}(Y_1\times Y_2)$ such that $F$ is invariant
under $S_1\times S_2$ can be written as
$$F=\sum_{i,j} c_{i,j}f_i\otimes \bar g_j.$$ for some constants
$c_{i,j}$ with $\sum_{i,j}|c_{i,j}|^2\leq \infty$. Write
$S_1f_i=e(t_i)f_i$, $S_2g_j=e(s_j)g_j$. By invariance of $F$ under
$S_1\times S_2$, we have
$$\sum_{i,j} c_{i,j}f_i\otimes \bar g_j=\sum_{i,j} c_{i,j}e(t_i-s_j)f_i\otimes \bar g_j.$$
Hence $c_{i,j}\neq 0$ only when $t_i=s_j$. This completes the proof.
\end{proof}

\begin{lemma}\label{chara1}
Let $\eta,\varphi,\psi$ be bounded functions on $X$. Then the averages
of the difference
$I_n(\eta,\varphi,\psi)-I_n(\eta,\widehat{\varphi},\widetilde{\psi})$ converge to
$0$.
\end{lemma}

\begin{proof}
Without loss of generality, we can assume that $\max\{|\eta|,|\varphi|,|\phi|\}\leq 1$.
By a density argument, we can suppose that
$\eta(y_1,y_2)=\alpha_1(y_1)\beta_1(y_2)$,
$\varphi(y_1,y_2)=\alpha_2(y_1)\beta_2(y_2)$, and that
$\psi(y_1,y_2)=\alpha_3(y_1)\beta_3(y_2)$ for some bounded functions
$\alpha_i$ on $Y_1$ and $\beta_i$ on $Y_2$, $i=1,2,3$, all of them being bounded by $1$ in absolute value. Using the notation introduced above, we have
$$\widehat \varphi(y_1,y_2)=\E(\alpha_2\mid\CZ)(y_1)\cdot\beta_2(y_2)\text{ and }\widetilde \psi(y_1,y_2)=\alpha_3(y_1)\cdot\E(\beta_3\mid\CZ)(y_2)\ .$$
For every $n$,
\begin{align*}
I_n(\eta,\varphi,\psi)-I_n(\eta,\widehat{\varphi},\widetilde{\psi})=\int (\alpha_1\alpha_3)(y_1)\cdot (\beta_1\beta_3)(y_2) \cdot \Big(\alpha_2(S_1^n y_1)\cdot\beta_3(S_2^n y_2)\\-\E(\alpha_2\mid\CZ)(S_1^n y_1)\cdot\E(\beta_3\mid\CZ)(S_2^n y_2)\Big)\, d\nu(y_1)d\nu(y_2)\ .
\end{align*}
By the Ergodic Theorem, the averages of this expression converge to  
\begin{eqnarray*}
 \E\Big( (\alpha_2-\E(\alpha_2\mid\CZ))\otimes \beta_3\mid \CI(S_1\times S_2)\Big)+\\ \E\Big( \E(\alpha_2\mid\CZ)\otimes (\beta_3-\E(\beta_3\mid\CZ))\mid \CI(S_1\times S_2)\Big)\ .
\end{eqnarray*}
By Lemma~\ref{claimprod}, this is equal to $0$. This completes
the proof.
\end{proof}

\begin{lemma}\label{chara2}
For every common eigenvalue $e(t)$ of
$Y_1$ and $Y_2$, the averages of the difference
$$e(nt)\big(I_n(f, f, f)-I_n(f,\widehat
f,\widetilde f)\big)$$ converge to $0$.
\end{lemma}

\begin{proof}
Let $e(t)$ be a common eigenvalue of $Y_1$ and $Y_2$, $h$ be a
corresponding eigenfunction of $Y_1$ with $ |h| = 1$. Taking
$\eta(y_1,y_2) = f(y_1,y_2)\bar h (y_1)$, $\varphi(y_1,y_2) =
f(y_1,y_2)h(y_1)$, and $\psi(y_1,y_2) = f(y_1,y_2)$, we have
$$e(nt)I_n(f, f, f)=I_n(\eta,\varphi,\psi)\ .$$ Since $\widehat \varphi(y_1,y_2)=h(y_1) \widehat f(y_1,y_2)$ and
$\widetilde \psi(y_1,y_2)=\widetilde f(y_1,y_2)$, we have $$e(nt)I_n(f, \widehat f, \widetilde f)=I_n(\eta,\widehat \varphi,\widetilde
\psi).$$ By Lemma~\ref{chara1}, the averages of the difference
$I_n(\eta,\varphi,\psi)-I_n(\eta,\widehat \varphi,\widetilde\psi)$ converge to $0$. This completes the proof.
\end{proof}

\subsection{A well chosen function $\chi$}\label{prodpart2}
In the sequel, we fix $f$, $\epsilon>0$ as in Theorem~\ref{prodf}. By
continuity of the translation in $L^2(Z)$, we can choose an
open neighborhood $U$ of $0$ in $Z$, such that
$$
\text{for every }s\in U, \int |\widehat f(z+s,y_2)-\widehat f(z,y_2)|^2
d\theta(z)d\nu_2(y_2)<\epsilon/4\ ,$$ and
$$
\text{for every }t\in U, \int |\widetilde f(y_1,z+t)-\widetilde f(y_1,z)|^2
d\nu_1(y_1)\theta(z)<\epsilon/4\ .
$$ We have
\begin{equation}\label{nainu}
|I_0(f,\widehat f,\widetilde f)-I_n(f,\widehat f,\widetilde f)|<\epsilon/2 \ \
\text{for every}\ n\  \text{such that}\ n\alpha\in U\ .
\end{equation}

We choose a continuous function $\chi$ on $Z$, non-negative,
satisfying
$$
\int \chi(z)d\theta (z)=1\ ,
$$
and
$$
\chi(z)=0\ \ \text{if}\ z\notin U.
$$

By Lemma~\ref{chara2} and a density argument, we immediately get
\begin{corollary}\label{chara3}
The averages of the
difference
$$\chi(n\alpha)\big(I_n(f,f,f)-I_n(f,\widehat
f,\widetilde f)\big)$$ converge to $0$.
\end{corollary}

We have
\begin{lemma}\label{bound}
\begin{equation}\label{bound1}
\limsup_{N-M\rightarrow\infty} \left|\frac{1}{N-M}\sum_{n\in
[M,N)}\chi(n\alpha)\big(I_0(f,\widehat f,\widetilde
f)-I_n(f,f,f)\big)\right|\leq \epsilon/2\ .
\end{equation}
\end{lemma}

\begin{proof}
By Lemma~\ref{chara3}, it suffices to prove
\begin{equation}\label{bound2}
\limsup_{N-M\rightarrow\infty} \left|\frac{1}{N-M}\sum_{n\in
[M-N)}\chi(n\alpha)\big(I_0(f,\widehat f,\widetilde f)-I_n(f,\widehat f,\widetilde
f)\big)\right|\leq \epsilon/2\ .
\end{equation}
If $n\alpha\notin U$, then $\chi(n\alpha)\big(I_0(f,\widehat f,\widetilde
f)-I_n(f,\widehat f,\widetilde f)\big)=0$. If $n\alpha\in U$, then
by~\eqref{nainu},$$\chi(n\alpha)|I_0(f,\widehat f,\widetilde f)-I_n(f,\widehat
f,\widetilde f)|<\chi(n\alpha)\cdot \epsilon/2\ .$$ By ergodicity of
$(Z,\theta,R)$, the sequence $\{n\alpha\}$ is uniformly distributed
in $Z$. Therefore the average of $\chi(n\alpha)$ converges
to $1$, and this gives~\eqref{bound2}.
\end{proof}

\subsection{End of the proof}
Applying Lemma~\ref{ineq} with $k=2$, $\CX_1=\CZ\times\CY_2$ and
$\CX_2=\CY_1\times \CZ$, we immediately get the following estimate of
$I_0(f, \widehat f, \widetilde f)$:
\begin{corollary}\label{prodineq}
$$I_0(f, \widehat f, \widetilde f)\geq (\int f\, d\mu)^3\ .$$
\end{corollary}

We can now resume the proof of Theorem~\ref{prodf}:
\begin{proof}[Proof of Theorem~\ref{prodf}]
By Corollary~\ref{prodineq}, it suffices to prove that the set $E\colon=\{n\colon I_n(f,f,f)>
I_0(f,\widehat f,\widetilde f)-\epsilon\}$ is syndetic. 

Suppose by contradiction that $E$ is
not syndetic, then there exists a sequence of intervals $[M_i,N_i)$
whose lengths tend to infinity such that $I_n(f)\leq
I_0(f,\widehat f,\widetilde f)-\epsilon$ for
every $i$ and for every $n\in [M_i,N_i)$. Taking averages of $\chi(n\alpha)\big(I_0(f,\widehat f,\widetilde
f)-I_n(f)\big)$ over these intervals, we have a contradiction with
Lemma~\ref{bound}.
\end{proof}

\section{Preliminaries: Magic systems}\label{sec:pre}
In this Section, $(X,\CX,\mu,T_1,T_2)$ is a system. We review some
definitions and properties of~\cite{H} needed in sequel, and
establish more properties that do not appear there.

The results of this section hold for any number of commuting
transformations, but we state and prove them only for two
transformations.

\subsection{The box measure and the box seminorm}\label{secseminorm}
Let $X^*=X^4$. The points of $X^*$ are written as
$x^*=(x_{00},x_{01},x_{10},x_{11})$.

\subsubsection{}
Let $\mu_1=\mu\times_{\CI (T_1)}\mu$ be the relatively independent
square of $\mu$ over $\CI(T_1)$. This means that $\mu_1$ is the
measure on $X^2$ characterized by:

For $f_0,f_1\in L^{\infty}(\mu)$, we have
$$\int f_0\otimes f_1\, d\mu_1=\int \E(f_0|\CI(T_1))\cdot\E(f_1|\CI(T_1))\, d\mu\ .$$
Then $\mu_1$ is invariant under the transformations $T_1\times T_1$, $T_1\times \id$ and $T_2\times T_2$. We define the measure $\mu^*$ on $X^*=X^4$ by
$$\mu^*=\mu_1\times_{\CI(T_2\times T_2)}\mu_1\ .$$
When needed, we write $\mu^*_{T_1,T_2}$ instead of $\mu^*$. Then $\mu^*$ is invariant under the transformations $T_1^*$ and $T_2^*$ given by $$T_1^*=T_1\times \id\times T_1\times \id\ ,$$
$$T_2^*=T_2\times T_2\times \id\times \id\ .$$

The projection $\pi_{00}\colon x\mapsto x_{00}$ is a factor map from $(X^*,\mu^*,T_1^*,T_2^*)$ to $(X,\mu,T_1,T_2)$.

We remark that by construction,
\begin{equation}\label{muchangetrans}
 \mu^*_{T_1,T_2}=\mu^*_{T_1^{-1},T_2}=\mu^*_{T_1,T_2^{-1}}\ .
\end{equation}

\subsubsection{}
We recall the definition of the seminorm introduced in~\cite{H}:

For every $f\in L^{\infty}(\mu)$ on $X$,
\begin{equation}
\label{eq:defseminorm} \nnorm f_{\mu,T_1,T_2}=\Big(\int
f(x_{00})f(x_{01})f(x_{10})f(x_{11})\, d\mu^*(x^*)\Big)^{1/4}\ .
\end{equation}
We generally omit the subscript $\mu$ when there is no ambiguity. By~\eqref{muchangetrans}, we have $$\nnorm f_{T_1,T_2}=\nnorm f_{T_1^{-1},T_2}\ ,$$ and by Proposition 2 of~\cite{H}, we have $$\nnorm f_{T_1, T_2}=\nnorm f_{T_2,T_1}\ .$$

Here are some properties that we need:
\begin{lemma}\label{upbound}
Let $f_0, f_1, f_2 \in L^{\infty}(\mu)$ with $|f_i|\leq 1$ $i=0,1,2$,  and for every $n$, we write
$$I_n(f_0,f_1,f_2):=\int f_0\cdot T_1^n f_1\cdot T_2^n f_2\,  d\mu\ .$$ Then for every $t\in\mathbb T$,
\begin{equation}\label{equpbound1}
\limsup_{N-M\rightarrow\infty}\left|\frac{1}{N-M}\sum_{n\in [M,N)}
e(nt)I_n(f_0,f_1,f_2)\right|\leq \nnorm{f_0}_{T_1,T_2}\ .
\end{equation}
\end{lemma}
For $t=0$, this is Proposition 1 of~\cite{H}, and the general case follows by the same proof.

Let $f_0,f_1,f_2$ and $I_n(f_0,f_1,f_2)$ be as in Lemma~\ref{upbound}. For every $n$, we have
$$I_n(f_0,f_1,f_2)=\int T_2^{-n}f_0\cdot (T_1T_2^{-1})^nf_1\cdot f_2\, d\mu\ .$$
Applying Lemma~\ref{upbound} to the pair of transformations $(T_2,
T_1T_2^{-1})$, we have
\begin{equation}\label{equpbound2}
\limsup_{N-M\rightarrow\infty}\left|\frac{1}{N-M}\sum_{n\in [M,N)}
e(nt)I_n(f_0,f_1,f_2)\right|\leq \nnorm{f_1}_{T_2,T_1T_2^{-1}}\ .
\end{equation}
and by the same argument,
\begin{equation}\label{equpbound3}
\limsup_{N-M\rightarrow\infty}\left|\frac{1}{N-M}\sum_{n\in [M,N)}
e(nt)I_n(f_0,f_1,f_2)\right|\leq \nnorm{f_2}_{T_1,T_1T_2^{-1}}\ .
\end{equation}

\begin{lemma}[\cite{H}, Lemma 1]\label{lem:seminorm}
For every $f\in L^{\infty}(\mu)$,
\begin{multline*}
\nnorm f_{T_1,T_2}=\\
\Big(\lim_{N_2\to+\infty} \frac{1}{N_2}\sum_{n_2=0}^{N_2-1} \bigl(
\lim_{N_1\to+\infty}\frac{1}{N_1}\sum_{n_1=0}^{N_1-1} \int f\cdot  T_1^{n_1} f \cdot T_2^{n_2}f\cdot
T_1^{n_1}T_2^{n_2}f\, d\mu\bigr)\Big)^{1/4}\ .
\end{multline*}
\end{lemma}
In fact, we can take the simultaneous instead of iterated limit by Theorem 1.1 of~\cite{C}.

By the H\"older Inequality, Lemma~\ref{lem:seminorm} implies that, for every $f\in
L^{\infty}(\mu)$,
\begin{equation}
\label{eq:semiL4} \nnorm f_{T_1,T_2}\leq\norm f_{L^4(\mu)}\ .
\end{equation}
Let $\pi\colon (X,\mu,T_1,T_2)\to (Y,\nu,S_1,S_2)$ be a factor map,
and $f\in L^{\infty}(\nu)$, by Lemma~\ref{lem:seminorm}, we have
\begin{equation}\label{normfactor}
 \nnorm {f\circ\pi}_{\mu,T_1,T_2} =\nnorm f_{\nu,S_1,S_2}\ .
\end{equation}
Moreover, by applying Lemma~\ref{lem:seminorm} to the measure $\mu$
and to each element of its ergodic decomposition, we get:

\begin{lemma}
\label{lem:semidecomp} If $\mu=\int \mu_x\,d\mu(x)$ is the ergodic
decomposition of $\mu$ under $T_1$ and $T_2$ then, for every $f\in
L^{\infty}(\mu)$,
$$
 \nnorm f_{\mu,T_1,T_2}^4=\int\nnorm f_{\mu_x,T_1,T_2}^4\,d\mu(x)\ .
$$
\end{lemma}

\subsubsection{}
We recall some properties about some $\sigma$-algebra on $X$
introduced in~\cite{H}:
\begin{lemma}[\cite{H}, Lemma 4]\label{lemmah4}
Let $\mathcal G$ be the
$\sigma$-algebra on $X$ consisting of sets $B$ such that there
exists a subset $A$ of $X^3$ satisfying the relation
\begin{equation}
 \mathbf{1}_B(x_{00})=\mathbf{1}_A(x_{01},x_{10},x_{11})
\end{equation}
for $\mu^*$-almost every $x=(x_{00},x_{01},x_{10},x_{11})\in X^4$.
Then for every $f\in L^{\infty}(\mu)$, we have
\begin{equation}
 \nnorm{f}_{T_1,T_2}=0\text{ if and only if } \E_{\mu}(f\mid\mathcal G)=0\ .
\end{equation}
\end{lemma}

\begin{lemma}\label{cz}
$$\mathcal G\supseteq \CI(T_1)\vee\CI(T_2)\ .$$
\end{lemma}
\begin{proof}
 Let $f\in L^{\infty}(\mu)$. If $f$ is invariant under $T_1$ then, by the construction of the measure $\mu^*$, we have
$f(x_{00})=f(x_{10})$ $\mu^*$-almost everywhere, and $f$ is
measurable with respect to $\mathcal G$. If $f$ is invariant under
$T_2$ then, by a similar reason, $f(x_{00})=f(x_{10})$
$\mu^*$-almost everywhere, and $f$ is measurable with respect
to $\mathcal G$. This proves the lemma.
\end{proof}

Combining Lemma~\ref{lemmah4} and Lemma~\ref{cz}, we immediately get:
\begin{corollary}\label{normimply}
Let $(X, \mu, T_1, T_2)$ be a system, then for every $f\in L^{\infty}(\mu)$, we have
$$
\text{if }\nnorm{f}_{T_1,T_2}=0\text{ then }\E(f\mid \CI(T_1)\vee\CI(T_2))=0\ .
$$
\end{corollary}

\subsection{Magic systems}
We recall the definition of the magic system:
\begin{definition}[\cite{H}]\label{defmagic}
A system $(X,\mu,T_1,T_2)$ is called a \textit{magic system} if for any
$f\in L^{\infty}(\mu)$ satisfying $\E(f|\CI(T_1)\vee\CI(T_2))=0$, we
have $\nnorm f_{T_1,T_2}=0$.
\end{definition}

\begin{corollary}[of Corollary~\ref{normimply}]\label{normequi}
If $(X, \mu, T_1, T_2)$ is a magic system, then
$$\nnorm{f}_{T_1,T_2}=0\text{ if and only if }\E(f\mid \CI(T_1)\vee\CI(T_2))=0\ .$$
\end{corollary}

The next Lemma plays an important role in Section~\ref{constr}.
\begin{lemma}\label{magextlem}
Let $p\colon (Y,\nu,S_1,S_2)\to (X,\mu,T_1,T_2)$ be a factor map and
assume that $(X,\mu,T_1,T_2)$ is magic. Then for every $f\in
L^\infty(\mu)$,
$$\nnorm{f\circ p-\E(f\circ p\mid \CI(S_1))\vee\CI(S_2)}_{S_1,S_2}=0\ .$$
\end{lemma}
\begin{proof}
 We write $f=f_1+f_2$, where $f_1$ is $\CI(T_1))\vee\CI(T_2)$-measurable, and $\E(f_2\mid \CI(T_1))\vee\CI(T_2))=0$. Since $(X,\mu,T_1,T_2)$ is magic, we have $\nnorm{f_2}_{T_1,T_2}=0$. By~\eqref{normfactor}, $\nnorm{f_2\circ p}_{S_1,S_2}=0$. By Lemma~\ref{normimply}, $\E(f_2\circ p\mid \CI(S_1)\vee\CI(S_2))=0$. On the other hand, $f_1\circ p$ is measurable with respect to $\CI(S_1)\vee\CI(S_2)$. Therefore $\E(f\circ p\mid \CI(S_1)\vee\CI(S_2))=f_1\circ p$ and $\nnorm{f\circ p-\E(f\circ p\mid \CI(S_1)\vee\CI(S_2))}_{S_1,S_2}=\nnorm{f_2\circ p}_{S_1,S_2}=0$.
\end{proof}

One of the main results of~\cite{H} is the following theorem:
\begin{theorem}[Theorem 1, \cite{H}]
\label{prop:magicext} The system $(X^*,\mu^*,T_1^*,T_2^*)$ built in Section~\ref{secseminorm} is magic.
\end{theorem}

Since $\pi_{00}\colon (X^*,\mu^*,T_1^*,T_2^*)\to(X,\mu,T_1,T_2)$ is
a factor map, we have:

\begin{corollary}
\label{cor:existsmagic} Every system $(X,\mu,T_1,T_2)$ admits a
magic extension.
\end{corollary}
We can not use this result directly because the system $X^*$ is not ergodic even if $X$ is
ergodic.

\subsection{The existence of ergodic magic extension}
In this Section we show:
\begin{theorem}
\label{th:existext} Every ergodic system $(X,\CX,\mu,T_1,T_2)$
admits an ergodic magic extension.
\end{theorem}
Recall that we assume that $(X,\CX,\mu)$ is a standard probability space. Then $X$ admits a dense countably generated $\sigma$-algebra $\CX'$. Substituting this algebra for $\CX$, we reduce to the case that the $\sigma$-algebra $\CX$ is countably
generated. In the sequel, 
$$
 \mu=\int \mu_x\,d\mu(x)
$$
denotes the ergodic decomposition of $\mu$ under $T_1$ and $T_2$.

Theorem~\ref{th:existext} follows from the following proposition:
\begin{proposition}\label{prop:general}
Let $(X,\CX,\mu,T_1,T_2)$ be a magic system.
Then for $\mu$-almost every $x$, the system $(X,\CX,\mu_x,T_1,T_2)$ is magic.
\end{proposition}

\begin{proof}[Proof of Theorem~\ref{th:existext} assuming
Proposition~\ref{prop:general}]

By Corollary~\ref{cor:existsmagic}, there exist
a magic system $(X^*,\mu^*,T_1^*,T_2^*)$ and a factor map
 $\pi\colon X^*\to X$. Let $\mu^*=\int
\mu^*_x\,d\mu^*(x)$ be the ergodic decomposition of $\mu^*$ under
$T_1^*$ and $T_2^*$. Then $\mu=\int
\pi(\mu^*_x)\,d\mu^*(x)$. Since the measures
$\pi(\mu_x^*)$ are invariant under $T_1$ and $T_2$ and since
$\mu$ is ergodic, then for $\mu^*$-almost every $x$ we have
$\pi(\mu_x^*)=\mu$, and thus
$\pi\colon(X^*,\mu^*_x,T_1^*,T_2^*)\to(X,\mu,T_1,T_2)$ is a
 factor map.
\end{proof}

Before giving the proof of Proposition~\ref{prop:general}, we need the following two lemmas. We begin with the notation. We denote
$$
 \CA=\CI(T_1)\wedge\CI(T_2)\text{ and }
\CW=\CI(T_1)\vee\CI(T_2)\ .
$$

We recall that for $\mu$-almost every $x$, the measure $\mu_x$ is
invariant and ergodic under $T_1$ and $T_2$. The map $x\mapsto\mu_x$
is $\CA$-measurable and for every bounded function $f$ on
$X$,
$$
\E_\mu(f\mid\CA)(x)=\int f\,d\mu_x\text{ for $\mu$-almost every
$x$.}
$$

\begin{lemma}
\label{lem}There exists a countable family $\Phi$ of bounded
$\CW$-measurable functions on $X$, which is dense in $L^p(\CW,\nu)$
for every $p\in [1,+\infty)$ and every probability measure $\nu$ on
$(X,\CX)$ which is invariant under $T_1$ and $T_2$.
\end{lemma}

\begin{proof}
Let $\CF=\{f_n\colon n\in \N\}$ be a countable family of bounded
$\CX$-measurable functions on $X$ which is dense in $L^p(\nu)$ for
every $p\in [1,+\infty)$ and every probability measure $\nu$ on
$(X,\CX)$. For each $n$, define
$$
 g_n(x)=\begin{cases}
\displaystyle\lim_{K\to\infty}\frac
1K\sum_{k=0}^{K-1} f_n(T_1^kx)
&\text{if this limit exists;}\\
0 &\text{otherwise.}
\end{cases}
$$
Then each function $g_n$ is bounded and invariant under $T_1$. For
every $n$, let $h_n$ be the function associated to $T_2$ by the same
construction.

Let $\CF_1= \{g_n\colon n\in\N\}$ and  $\CF_2= \{h_n\colon
n\in\N\}$. Then, for $i=1,2$, $\CF_i$ is dense in
$L^p(\CI(T_i),\nu)$  for every $p\in [1,+\infty)$ and every
$T_i$-invariant probability measure $\nu$ on $(X,\CX)$. Let $\Phi$
be family of functions consisting in finite sums of functions of the
form $gh$ with $g\in\CF_1$ and $h\in\CF_2$, then $\Phi$ satisfies
the lemma.
\end{proof}

\begin{lemma}\label{expect}
Let $f$ be a bounded function on $X$ and let $\widetilde
f$ be a $\CW$-measurable representative of $\E_\mu(f\mid\CW)$.
 Then, for
$\mu$-almost every $x$, we have
$$\E_{\mu_x}(f\mid\CW)=\widetilde f\quad\mu_x\text{-almost everywhere.}
$$
\end{lemma}

\begin{proof}
Let $\Phi=\{\phi_n\colon n\in\N\}$ be the family of functions given
by Lemma~\ref{lem}. For every $n$, $\E_\mu((f-\widetilde
f)\phi_n\mid\CW)=0$ and, since $\CA\subset\CW$,
$\E_\mu((f-\widetilde f)\phi_n\mid\CA)=0$ $\mu$-almost everywhere.
That is, $\int (f-\widetilde f)\phi_n d\mu_x=0$ for $\mu$-almost
every $x$. Thus there exists a set $A_n$ with $\mu(A_n)=1$ such that
$\int (f-\widetilde f)\phi_n\,d\mu_x= 0$ for every $x\in A_n$.

Let $A$ be the intersection of the sets $A_n$ for $n\in\N$. Then
$\mu(A)=1$. Let $x\in A$. We have $\int (f-\widetilde
f)\phi_n\,d\mu_x= 0$ for every $n$ and, by density of the family
$\{\phi_n\colon n\in\N\}$ in $L^1(\CW,\mu_x)$, we have
$\E_{\mu_x}(f-\widetilde f\mid\CW)=0$. Since $\widetilde f$ is
measurable with respect to $\CW$, we have $\E_{\mu_x}(f\mid\CW)=\widetilde f$.
\end{proof}

\begin{proof}[Proof of Proposition \ref{prop:general}]
Let $\CF=\{f_n\colon n\in \N\}$ be a countable family of bounded
$\CX$-measurable functions on $X$ which is dense in $L^p(\nu)$ for
every $p\in [1,+\infty)$ and every probability measure $\nu$ on
$(X,\CX)$.

For each $n$, let $\widetilde f_n$ be  a
$\CW$-measurable  representative of
$\E_{\mu}(f_n\mid\CW)$. By Lemma~\ref{expect}, for every $n$ we have
that $\widetilde f_n=\E_{\mu_x}(f_n\mid\CW)$ $\mu_x$-almost
everywhere for $\mu$-almost every $x$, and there exists a set
$A_n$ such that $\mu(A_n)=1$ and,
$$
 \text{for every $x\in A_n$, }\quad
 \E_{\mu_x}(f_n\mid\CW)=\widetilde f_n
\quad \mu_x\text{-a.e.}
$$

Fix an integer $n\in\N$. Since $\E_{\mu}(f_n-\widetilde
f_n\mid\CW)=0$ and $(X,\mu,T_1,T_2)$ is magic, we have $\nnorm{
f_n-\widetilde f_n}_{\mu,T_1,T_2}=0$.

By Lemma~\ref{lem:semidecomp}, we have
$$
0= \nnorm{ f_n-\widetilde f_n}_{\mu,T_1,T_2}^4=\int \nnorm {f_n-\widetilde
f_n}_{\mu_x,T_1,T_2}^4\,d \mu(x)\ .
$$
So $\nnorm {f_n-\widetilde f_n}_{\mu_x,T_1,T_2}=0$ for $\mu$-almost every
$x$. There exists a set $B_n\subset X$ such that $\mu(B_n)=1$ and
$$
\text{for every }x\in B_n,\quad
 \nnorm{f_n-\widetilde f_n}_{\mu_x,T_1,T_2}=0\ .
$$
Define
$$
 C=\bigcap_{n\in N}(A_n\cap B_n)\ .
$$
We have $\mu(C)=1$. We check that for every $x\in C$ the system
$(X,\mu_x,T_1,T_2)$ is magic.

Let $x\in C$ and $f$ be a bounded measurable function on $X$, we
have to show that $\nnorm{f-\E_{\mu_x}(f\mid\CW)}_{\mu_x,T_1,T_2}=0$.

Since the family $\{f_n\}$ is dense in $L^4(\mu_x)$, there exists a
sequence $(n_i)$ of integers such that $f_{n_i}$ converges to $f$ in
$L^4(\mu_x)$. By~\eqref{eq:semiL4} applied to the measure $\mu_x$, we have
\begin{equation}\label{pgeq1}
\nnorm{f_{n_i}-f}_{\mu_x,T_1,T_2}\to 0\ .
\end{equation}

On the other hand, since $f_{n_i}$ converges to $f$ in $L^4(\mu_x)$, we have
that $\E_{\mu_x}(f_{n_i}\mid\CW)$ converges to $\E_{\mu_x}(f\mid\CW)$ in
$L^4(\mu_x)$ and, as above,
\begin{equation}\label{pgeq2}
\nnorm{  \E_{\mu_x}(f_{n_i}\mid\CW)-
\E_{\mu_x}(f\mid\CW)}_{\mu_x,T_1,T_2}\to 0 \ .
\end{equation}

For every $i$, since $x\in A_{n_i}$  we have that
$\E_{\mu_x}(f_{n_i}\mid\CW)=\widetilde f_{n_i}$ ($\mu_x$-a.e.).
Since $x\in B_{n_i}$, we have $\nnorm{ f_{n_i}-\widetilde
f_{n_i}}_{\mu_x,T_1,T_2}=0$ and thus
\begin{equation}\label{pgeq3}
\nnorm{f_{n_i}-\E_{\mu_x}(f_{n_i}\mid\CW)}_{\mu_x,T_1,T_2}=0\ .
\end{equation}

For every $i$,
\begin{multline*}
\nnorm{f-\E_{\mu_x}(f\mid\CW)}_{\mu_x,T_1,T_2}\\
\leq\nnorm{f_{n_i}-f}_{\mu_x,T_1,T_2}+\nnorm{ \E_{\mu_x}(f_{n_i}\mid\CW)-
\E_{\mu_x}(f\mid\CW)}_{\mu_x,T_1,T_2}\\+\nnorm{f_{n_i}-\E_{\mu_x}(f_{n_i}\mid\CW)}_{\mu_x,T_1,T_2}\ .
\end{multline*}
Combining (\ref{pgeq1}), (\ref{pgeq2}), (\ref{pgeq3}), we get
$\nnorm{f-\E_{\mu_x}(f\mid\CW)}_{\mu_x,T_1,T_2}=0$ for every $x\in C$. This
finishes the proof.
\end{proof}

\section{Proof of Theorem~\ref{mmaina}}\label{sec:general}
In this section, we prove:

\begin{theorem}\label{mmainf}
Let $(X,\mu,T_1,T_2)$ be an ergodic system and $0\leq f\leq 1$. Then
for every $\epsilon>0$, the set
$$\{n\in\ZZ\colon\ \int f\cdot T_1^n f\cdot T_2^n f\, d\mu> (\int f d\mu)^4-\epsilon\}$$
is syndetic.
\end{theorem}

Theorem~\ref{mmaina} follows from Theorem~\ref{mmainf} with $f=\mathbf{1}_A$.

In this section, we fix the ergodic system $(X,\mu, T_1,T_2)$, and the function $f$ as in Theorem~\ref{mmainf}, and for every $n$, we write
$$I_n\colon=\int f\cdot T_1^n f\cdot T_2^n f\, d\mu\ .$$

\subsection{Construction of ergodic magic extensions}\label{constr}
We write $T_3=T_1T_2^{-1}$. We use Theorem~\ref{th:existext} three times:

Let $(X',\mu',T_1',T_2')$ be an ergodic magic extension of $(X,\mu,
T_1,T_2)$ and let $\pi'\colon X'\to X$ be the factor map. Write
$T_3'=T_1'T_2'^{-1}$. Then the system $(X',\mu',T_1',T_3')$ is ergodic.

Let $(X'',\mu'',T_1'',T_3'')$ be an ergodic magic extension of
$(X',\mu', T_1',T_3')$ and let $\pi''\colon X''\to X'$ be the
factor map. Write $T_2''=T_1''T_3''^{-1}$. Then the system
$(X'',\mu'',T_2'',T_3'')$ is ergodic.

Let $(\overline X,\bar \mu,\overline T_2,\overline T_3)$ be an ergodic magic extension of
$(X'',\mu'', T_2'',T_3'')$ and let $\bar{\pi}\colon \overline X\to X''$ be the
factor map. Write $\overline T_1=\overline T_2\overline T_3$. Then the three systems $(\overline X,\bar \mu,\overline T_2,\overline T_3)$, $(\overline X,\bar \mu,\overline T_1,\overline T_3)$ and
$(\overline X,\bar \mu,\overline T_1,\overline T_2)$ are all ergodic, and  $\pi'\circ\pi''\circ\bar{\pi}$ is a factor map from $(\overline X,\bar \mu,\overline T_1,\overline T_2,\overline T_3))$ to $(X,\mu,T_1,T_2,T_3)$.

In the sequel, we write
\begin{align}\label{hggg}
 \begin{split}
h&=\colon f\circ \pi'\circ\pi''\circ \bar{\pi}\,\\
g_0&=\colon\E(h\mid\CI(\overline{T}_1)\vee\CI(\overline{T}_2))\ ,\\
g_1&=\colon\E(h\mid \CI(\overline {T}_1)\vee\CI(\overline{T}_3))\ ,\\
g_2&=\colon\E(h\mid \CI(\overline{T}_2)\vee\CI(\overline{T}_3))\ .
 \end{split}
\end{align}

For every $n$, and all bounded functions $\eta,\varphi,\psi$ on
$\overline{X}$, we define
$$\bar{I}_n(\eta,\varphi,\psi)\colon=\int \eta \cdot
\overline{T}^{n}_1 \varphi \cdot \overline{T}^{n}_2 \psi\, d\bar{\mu}\
.$$ In particular, 
\begin{equation}\label{in}
 \bar{I}_n(h,h,h)=I_n\ .
\end{equation}
We write also
$$J_n\colon=\bar{I}_n(g_0,g_1,g_2)\ .$$
We have

\begin{proposition}\label{diff}
Let $t\in \mathbb T$. Then the averages of the difference
$e(nt)\big(I_n-J_n\big)$ converge to $0$.
\end{proposition}
\begin{proof}
Since $(\overline{X},\bar{\mu},\overline{T}_2,\bar{T}_3)$ is magic, by
Definition~\ref{defmagic} of magic systems and the definition of $g_2$,
$$\nnorm{h-g_2}_{\overline{T}_2,\overline{T}_3}=0\ .$$
Then by Lemma~\ref{upbound}, the averages of the difference of
$$e(nt)\big(\bar{I}_n(h,h,h)-\bar{I}_n(h,h,g_2)\big)$$ converge to
$0$. Since $(X'',\mu'',T_1'',T_3'')$ is magic, by Lemma~\ref{magextlem},
$$\nnorm{h-g_1}_{\overline{T}_1,\overline{T}_3}=0\ .$$
Then by Lemma~\ref{upbound}, the averages of the difference of
$$e(nt)\big(\bar{I}_n(h,h,g_2)-\bar{I}_n(h,g_1,g_2)\big)$$ converge
to $0$. Since $(X',\mu',T_1',T_2')$ is magic, by the same argument
as above, the averages of the difference of
$$e(nt)\big(\bar{I}_n(h,g_1,g_2)-J_n\big)$$ converge to $0$. Summing up
all the results above, we get the announced result.
\end{proof}

In the sequel, we write $$\overline{\mathcal
M}\colon=\CI(\overline{T}_1)\vee
\CI(\overline{T}_2)\vee\CI(\overline{T}_3)\ ,$$ and $$g\colon = \E(h\mid
\overline{\mathcal M})\ .$$ We remark that
\begin{align*}
\E(g\mid\CI(\overline{T}_1)\vee\CI(\overline{T}_2))=g_0\ ,\\ 
\E(g\mid\CI(\overline{T}_1)\vee\CI(\overline{T}_3))=g_1\ , \\
\E(g\mid\CI(\overline{T}_2)\vee\CI(\overline{T}_3))=g_2\ .
\end{align*}
Furthermore
\begin{equation}\label{eq:integraleq}
 \int g d\bar{\mu}=\int h d\bar{\mu}=\int f d\mu\ .
\end{equation}

\subsection{Reduction to a special system}
For $i=1,2,3$, let $(Y_i,\mathcal Y_i, \nu_i)$ be a factor of $\overline{X}$ which is
isomorphic to $(\overline{X},\CI(\overline{T}_i),\bar{\mu})$. Denote by $\pi_i:
\overline{X}\rightarrow Y_i$ the factor map. Then $\nu_i:=\pi_i(\bar{\mu})$.

Since for $j=1,2,3$, every $\sigma$-algebra $\CI(\overline{T}_i)$ is
invariant under the transformation $\overline{T}_j$, then each of these transformations induces a
measure preserving transformation on each factor $Y_i$.

The transformation $\overline{T}_1$ induces the identity on $Y_1$ and the transformations $\overline{T}_2$ and $\overline{T}_3$ induce the same transformation on $Y_1$, which we write $R_1$. Therefore we have
$$R_1\circ\pi_1=\pi_1\circ \overline{T}_2=\pi_1\circ \overline{T}_3\ .$$
By ergodicity of $(\overline{X},\overline{\CX},\bar{\mu},\overline{T}_1,\overline{T}_2)$, the system $(\overline{X},{\CI(\overline{T}_1)},\bar{\mu},\overline{T}_2)$ is ergodic and thus $(Y_1,\nu_1,R_1)$ is ergodic.

In the same way, we get a measure preserving transformation $R_2$ on
$Y_2$ and a measure preserving transformation $R_3$ on $Y_3$ with
$$R_2\circ\pi_2=\pi_2\circ\overline{T}_1=\pi_2\circ\overline{T}_3^{-1}\ ,$$
$$R_3\circ\pi_3=\pi_3\circ\overline{T}_1=\pi_3\circ\overline{T}_2\ .$$

As above, we have $(Y_2,\nu_2,R_2)$ and $(Y_3,\nu_3,R_3)$ are ergodic.

Let $Y=Y_1\times Y_2\times Y_3$ and $\CY$ be the product
$\sigma$-algebra on $Y$. Define $\pi: (\overline{X},\overline{\mathcal M})\rightarrow (Y,\CY)$ by
$\pi=\pi_1\times\pi_2\times\pi_3$. Let $\nu$ be the image of $\mu$ under $\pi$. Define
\begin{align}\label{sdef}
 \begin{split}
S_1&=\id\times R_2\times R_3\ ,\\
S_2&=R_1\times \id\times R_3\ ,\\
S_3&=R_1\times R_2^{-1}\times \id\ .
 \end{split}
\end{align}

We have
\begin{lemma}\label{equivatoy}
 $\pi$ is an isomorphism from $(\overline{X},\overline{\mathcal M},
\bar{\mu},\overline{T}_1,\overline{T}_2,\overline{T}_3)$ to 
\newline $(Y,\CY,\nu,S_1,S_2,S_3)$.
\end{lemma}

We can thus identify these two systems and consider $g$ as a
function on $Y$. Then we have 
\begin{align*}
g_0=\E(g\mid Y_1\times Y_2)\ ,\\
g_1=\E(g\mid Y_1\times Y_3)\ ,\\
g_2=\E(g\mid Y_2\times Y_3)\ .
\end{align*}
By~\eqref{eq:integraleq}, we have
\begin{equation}\label{eq:integraleq1}
 \int g d\nu =\int f d\mu\ .
\end{equation}

\subsection{Description of the special system}\label{descriy}
We recall the properties that will be used in the sequel:
\begin{enumerate}[(a)]
 \item For $i=1,2,3$, $(Y_i,\CY_i,\nu_i,R_i)$ is an ergodic system.
 \item $Y=Y_1\times Y_2\times Y_3$ is endowed with the product $\sigma$-algebra $\CY$. The projection $Y\to Y_i$ is written $\pi_i$.
 \item $\nu$ is a joining of $\nu_1$, $\nu_2$ and $\nu_3$, meaning that, for $i=1,2,3$, the projection of $\nu$ on $Y_i$ is equal to $\nu_i$. Moreover $\nu$ is invariant under the transformations $S_1, S_2, S_3$ defined in~\eqref{sdef}.
 \item\label{indentialg} For $i=1,2,3$, $\CY_i=\CI(S_i)$.
 \item\label{yerg} $(Y,\CY,\nu,S_1,S_2)$ is ergodic.
\end{enumerate}

\begin{remark}
For $i\neq j$, $Y$ is ergodic under $S_i$ and $S_j$.
By~\eqref{indentialg}, the $\sigma$-algebra $\CY_i$ and $\CY_j$ are
independent. In other words, the projection of $\nu$ on $Y_i\times
Y_j$ is equal to $\nu_i\times\nu_j$.
\end{remark}

The following proposition is a more precise description of $\nu$:
\begin{proposition}\label{nudescri}
Let $(Z,\theta, R)$ be the common factor of the systems
$Y_1,Y_2,Y_3$ spanned by the eigenfunctions corresponding to the
common eigenvalues of these three systems. Then for any $\alpha_1\in L^{\infty}(Y_1)$,
$\alpha_2\in L^{\infty}(Y_2)$ and $\alpha_3\in L^{\infty}(Y_3)$, we
have
\begin{equation}\label{nudef}
\int \alpha_1(y_1)\alpha_2(y_2)\alpha_3(y_3)\, d\nu=\int
\E(\alpha_1\mid Z)(z_1)\E(\alpha_2\mid Z)(z_2)\E(\alpha_3\mid
Z)(z_3)\, d m
\end{equation}
where $m$ is the image of $\nu$ on $Z\times Z\times Z$. In other words,
$\nu$ is the conditionally independent product measure over $Z\times
Z\times Z$.
\end{proposition}

\begin{proof}
For $i=1,2,3$ let $\K_i$ be the Kronecker factor of $Y_i$. Since $\nu$ is invariant under $S_1$,  we have
\begin{equation*}
\int \alpha_1(y_1)\alpha_2(y_2)\alpha_3(y_3)\, d\nu
=\frac{1}{N}\sum_{n=1}^N \int \alpha_1(y_1)\cdot R_2^n
\alpha_2(y_2)\cdot R_3^n \alpha_3(y_3)\, d\nu\ ,
\end{equation*}
for every $N$. Since $Y_2$, $Y_3$ are independent, taking the limit as $N\to\infty$, we have
$$\int \alpha_1(y_1)\alpha_2(y_2)\alpha_3(y_3)\, d\nu=\int \alpha_1(y_1)\cdot \E_{\nu_2\times \nu_3}(\alpha_2\otimes
\alpha_3\mid\CI(R_2\times R_3))(y_2,y_3)\, d\nu\ .$$ Note that the
functions $\alpha_2\otimes \alpha_3$ and $\E(\alpha_2\mid
\K_2)\otimes \E(\alpha_3\mid \K_3)$ have the same conditional
expectation on $\CI_{\nu_2\times\nu_3}(R_2\times R_3)$. Therefore we
have
$$
\int \alpha_1(y_1)\alpha_2(y_2)\alpha_3(y_3)\, d\nu =\int
\alpha_1(y_1)\E(\alpha_2\mid\K_2)(y_2)\E(\alpha_3\mid\K_3)(y_3)\, d\nu\
.
$$

By the same computation, substituting $S_2$ for $S_1$, we obtain
\begin{align}\label{nudescri1}
\int \alpha_1(y_1)\alpha_2(y_2)\alpha_3(y_3)\, d\nu =\int
\E(\alpha_1\mid\K_1)(y_1)\E(\alpha_2\mid\K_2)(y_2)\E(\alpha_3\mid\K_3)(y_3)\, d\nu\ .
\end{align}

On the other hand, for $i=1,2,3$, the right hand side of~\eqref{nudef} remains unchanged if $\E(\alpha_i\mid \K_i)$ is substituted for $\alpha_i$. Therefore we can reduce to the case that $\alpha_i$ is measurable with respect to $\K_i$. By density, we can suppose that $\alpha_i$ is an
eigenfunction of $Y_i$ corresponding to the eigenvalue $\lambda_i$. By invariance of $\nu$ under $S_1$ again, we get
$$\int
\alpha_1(y_1)\alpha_2(y_2)\alpha_3(y_3)\, d\nu=e(\lambda_2+\lambda_3)
\int \alpha_1(y_1)\alpha_2(y_2)\alpha_3(y_3)\, d\nu\ ,
$$
thus $\int \alpha_1(y_1)\alpha_2(y_2)\alpha_3(y_3)\, d\nu=0$ except if
$\lambda_2=-\lambda_3$.

By the same argument, we have that $\int
\alpha_1(y_1)\alpha_2(y_2)\alpha_3(y_3)\, d\nu=0$ except if
$\lambda_1=\lambda_2=-\lambda_3$. The announced conclusion follows.
\end{proof}

We have the following description of the measure $m$:
\begin{proposition}\label{mdescri}
The measure $m$ is the Haar measure on the compact subgroup $H$ of
$Z\times Z\times Z$, where
$$H=\{(z_1,z_2,z_3), z_1+z_2-z_3=0\}\ .
$$
\end{proposition}
\begin{proof}
Let $\alpha\in Z$ be the element defining the translation $R$. Since
$\nu$ is invariant and ergodic under the transformations $S_1$ and
$S_2$, then the measure $m$ is invariant and ergodic under the
transformations $\id\times R\times R$ and $R\times \id\times R$,
that is under the translations by $(0,\alpha,\alpha)$ and
$(\alpha,0,\alpha)$. The function $z_1+z_2-z_3$ is invariant under
these translations and by ergodicity it is equal to a constant
$c$.

Since the map $z\mapsto z-c$ is an isomorphism from $(Z,\theta, R)$
to itself, we can thus reduce to the case that $c=0$ and 
$m$ is concentrated on $H$. Since the subset $\ZZ \alpha$ is dense in $Z$, $m$ is
invariant under the translations by $(0,z,z)$ and $(z',0,z')$ for
any $z,z'\in Z$. Since these elements span $H$, $m$ is invariant
under translation by any element of $H$, and thus $m$ is the Haar
measure on $H$.
\end{proof}

\subsection{Study of $J_n$}
In this section, let $(Y,\nu,S_1,S_2)$ be as described in
Section~\ref{descriy}. We keep using the notation $g,g_0,g_1,g_2$
and $J_n=\bar{I}_n(g_0,g_1,g_2)$ introduced above. From this point, the proof is parallel to that in Section~\ref{prodpart2}.

\begin{proposition}\label{existu}
For every $\epsilon>0$, there exists an open neighborhood $U$ of $0$
in $Z$, such that
$$
 |J_0-J_n|<\epsilon\ .
$$
for every $n$ such that $n\alpha\in U$.
\end{proposition}
\begin{proof}
More generally, we show for all bounded functions $h_0$ on $Y_1\times Y_2$
$Y$, $h_1$ on $Y_1\times Y_3$, and $h_2$ on $Y_2\times Y_3$, that there exists an open neighborhood $U$ of $0$ in $Z$, such
that
$$
 \big|\int h_0\cdot h_1\cdot h_2\, d\nu-\int h_0\cdot S_1^n h_1\cdot S_2^n h_2\, d\nu\big|<\epsilon
$$
for every $n$ such that $n\alpha\in U$.

Without loss of generality, we can suppose that
$\max\{|h_0|,|h_1|,|h_2|\}\leq 1$. We first suppose that
$h_0,h_1,h_2$ are product functions:
$h_0(y_1,y_2)=\alpha(y_1)\beta(y_2)$,
$h_1(y_1,y_3)=\alpha'(y_1)\gamma(y_3)$, and
$h_2(y_2,y_3)=\beta'(y_2)\gamma'(y_3)$, for some functions $\alpha,\alpha'$ on $Y_1$, $\beta,\beta'$ on $Y_2$, and
$\gamma,\gamma'$ on $Y_3$, all of them being bounded by $1$ in
absolute value. Then, for every $n$
$$\int h_0 \cdot S_1^n h_1 \cdot S_2^n h_2 d\nu=\int(\alpha\alpha')(y_1)\cdot (\beta\beta')(y_2)\cdot (\gamma\gamma')(R_3^n y_3)\,
d\nu(y_1,y_2,y_3)\ .$$
By Proposition~\ref{nudescri}, this is equal to
$$\int \E(\alpha\alpha'\mid
Z)(z_1)\E(\beta\beta'\mid Z)(z_2)\E(\gamma\gamma'\mid
Z)(z_3+n\alpha)\, d m(z_1,z_2,z_3)\ .
$$Set $$U=\{t, \int \left|\E(\gamma\gamma'\mid
Z)(z_3+t)-\E(\gamma\gamma'\mid Z)(z_3)\right|^2 d\theta
(z_3)<\epsilon\},$$ then $U$ is an open neighborhood of $0$ in $Z$
and$ |J_0-J_n|<\epsilon \text{ when } n\alpha\in U$.

By linearity, the neighborhood $U$ with the announced property exists when each $h_i$, $i=1,2,3$, is a finite sum of product functions of the type described above. The general case follows by a density argument.
\end{proof}

Let $\epsilon>0$ be as in Theorem~\ref{mmainf}, and let $U$ be the
neighborhood of $0$ in $Z$ defined by Proposition \ref{existu} with
$\epsilon/2$ substituted for $\epsilon$. Let $\chi$ be a
non-negative continuous function on $Z$, such that
$$\int\chi d\theta=1\ ,$$ and $$\chi(z)=0 \text{ if } z\notin U\ .$$
We keep the notation $\epsilon$, $U$, $\chi$ in the sequel.

By Proposition~\ref{diff} and a density argument, we have
\begin{corollary}\label{cordiff}
The averages of the difference of
$$\chi(n\alpha)\big(I_n-J_n\big)$$
converge to zero.
\end{corollary}

\begin{lemma}\label{diffiny}
\begin{equation}\label{epsineig}
\limsup_{N-M\rightarrow \infty}\left|\frac{1}{N-M}\sum_{n\in
[M,N)}\chi(n\alpha)\big(J_0-J_n\big)\right|\leq\frac{\epsilon}{2}\ .
\end{equation}
\end{lemma}
\begin{proof}
If $n\alpha\notin U$, then $\chi(n\alpha)\big(J_0-J_n\big)=0$. If
$n\alpha\in U$, then by Proposition \ref{existu},
$$\chi(n\alpha)|J_0-J_n|<\frac{\epsilon}{2}\cdot\chi(n\alpha)\
.$$  Since $\{n\alpha\}$ is uniformly distributed in $Z$, the
averages of $\chi(n\alpha)$ converge to $1$, this gives
(\ref{epsineig}).
\end{proof}

\subsection{End of the proof}
We have the following estimate of $J_0$:
\begin{corollary}[of Lemma~\ref{ineq}]\label{j0}
 $$J_0\geq (\int g d\nu)^4=(\int f d\mu)^4 \ .$$
\end{corollary}
\begin{proof}
Since $0\leq g\leq 1$, we have
$$J_0=\int g_0\cdot g_1\cdot g_2\, d\nu\geq \int g\cdot g_0\cdot g_1\cdot g_2\, d\nu\ .$$
Applying Lemma~\ref{ineq} to the probability space $(Y,\nu)$ with $k=3$, $\CX_1=Y_1\times Y_2$,
$\CX_2=Y_1\times Y_3$ and $\CX_3=Y_2\times Y_3$, the announced conclusion follows.
\end{proof}

We can now resume the proof of Theorem~\ref{mmainf}:
\begin{proof}[Proof of Theorem~\ref{mmainf}]
By Corollary~\ref{j0}, it suffices to prove that the set $E\colon =\{n\in \ZZ\colon
I_n>J_0-\epsilon\}$ is syndetic. Suppose by contradiction that $E$ is not syndetic. Then there
exists a sequence of intervals $[M_i,N_i)$ whose lengths $N_i-M_i$
tend to infinity and such that $I_n\leq J_0-\epsilon$ for every $i$ and every $n\in
[M_i,N_i)$. Therefore
\begin{equation}\label{eq:proofmmainf}
\frac{1}{N_i-M_i}\sum_{n\in
[M_i,N_i)}\chi(n\alpha)\big(J_0-I_n\big)\geq
\frac{1}{N_i-M_i}\sum_{n\in [M_i,N_i)}\chi(n\alpha)\cdot\epsilon
\end{equation}
for every $i$. Taking the $\limsup$ of both sides, since the
averages of $\chi(n\alpha)$ converge to $1$, we get that the
$\limsup$ of the left hand side of~\eqref{eq:proofmmainf} is greater
than or equal to $\epsilon$.

On the other hand, it follows from Corollary~\ref{cordiff} and
Lemma~\ref{diffiny} that the $\limsup$ of the left hand side
of~\eqref{eq:proofmmainf} is less than or equal to $\epsilon/2$. We
have a contradiction.
\end{proof}

\section{Proof of Theorem~\ref{counterex}}\label{sec:ex}
Let $(Y, \nu, S)$ be the two-side
$(\frac{1}{3},\frac{1}{3},\frac{1}{3})$-Bernoulli-shift. This means
that $\{0,1,2\}$ is endowed with the uniform probability measure
$\tau$, and that $Y=\{0,1,2\}^{\ZZ}$ is endowed with the product measure $\nu=\tau^{\ZZ}$. The points of $Y$ are written $y =(y_{n}\colon n\in
\mathbb Z)$ where $y_{n}\in \{0,1,2\}$ for every $n$. The shift map
$S$ is given by $(Sy)_n=y_{n+1}$ for every $n$.

Define $$(X,\mu,T_1,T_2)=(Y\times Y\times Y, \nu\times \nu\times\nu,
S\times \id\times S, \id\times S\times S).$$ Since $(Y, \nu, S)$ is
weakly mixing, the system $(X,\mu,T_1,T_2)$ is ergodic.

For $(i,j,k)\in \{0,1,2\}^3$, let $f(i,j,k)$ be defined 
by the following table:

\begin{center}
\begin{tabular}{|c|c|c|c|}
\hline
 & i=0 & i=1 & i=2\\
\hline
j=0 & 0\ 0\ 1 & 0\ 0\ 1 & 1\ 1\ 1 \\
j=1 & 0\ 1\ 1 & 0\ 0\ 1 & 1\ 1\ 0 \\
j=2 & 1\ 1\ 0 & 1\ 1\ 1 & 1\ 0\ 0 \\
\hline
k=  & 0\ 1\ 2 & 0\ 1\ 2 & 0\ 1\ 2\\
\hline
\end{tabular}
\end{center}

Let $F$ be the function on $Y\times Y\times Y$ defined by $F( y, z, w)=f(y_0,z_0,w_0)$. Then $F$ is an indicator function of some measurable subset $A\subset
X$. For every $n\neq 0$, we have
\begin{eqnarray*}
& &\mu(A\cap T_1^n A\cap T_2^n A)=\int F\cdot T_1^{-n} F\cdot
T_2^{-n} F\, d\mu\\&=&\int F(y, z, w)\cdot F(S^{-n} y,z,S^{-n} w)\cdot
F( y, S^{-n} z,S^{-n} w)\, d\mu( y, z, w)\\&=&\int f(y_0,z_0,w_0)\cdot
f(y_{-n},z_0,w_{-n})\cdot f(y_0,z_{-n},w_{-n})\, d\mu( y, z, w).
\end{eqnarray*}

Since $n\neq 0$, the variables at the position $0$ are independent
with those at the position $-n$, we have
\begin{eqnarray*}
& &\mu(A\cap T_1^n A\cap T_2^n A)\\&=&\int
f(y,z,w')f(y',z,w)f(y,z',w)\, d\tau(y)d\tau(y')d\tau(z)d\tau(z')d\tau(w)d\tau(w')\\
&=&\frac{1}{3^6}\sum_{i,j,k,i',j',k'}f(i,j,k')f(i,j',k)f(i,j',k)\
.\end{eqnarray*} By computation, for every integer $n\neq 0$, we
have
\begin{eqnarray*}
 \mu(A\cap T_1^n A\cap T_2^n A)=\frac{145}{729}<0.96(\frac{16}{27})^3\\=
0.96(\frac{1}{3^3}\sum_{i,j,k}f(i,j,k))^3=0.96(\mu(A))^3\ .
\end{eqnarray*}

Let $0<c\leq 1$. There exists an integer $l>0$ such that $(0.96)^l<
c$. Let $(X',\mu',T_1',T_2')$ be the $l$-times product system of
$(X,\mu,T_1,T_2)$, and $A'$ be the $l$-times product of set $A$. Then the system $(X',\mu',T_1',T_2')$ is ergodic and the set $A'$
satisfies the announced property.

\bibliographystyle{amsplain}

\end{document}